\documentclass[11pt, reqno]{amsart}
\usepackage{xcolor}
\usepackage[title,titletoc,toc]{appendix}

\usepackage{amsmath}
\usepackage{amssymb}
\usepackage{amsfonts}
\usepackage{mathrsfs}
\usepackage{amsbsy}
\usepackage{relsize}

\newcommand{\beq}{\begin{eqnarray}}
\newcommand{\eeq}{\end{eqnarray}}
\newcommand{\bq}{\begin{equation}}
\newcommand{\eq}{\end{equation}}
\newcommand{\beqn}{\begin{eqnarray*}}
\newcommand{\eeqn}{\end{eqnarray*}}

\newcommand{\Lam}{\Lambda}

\newcommand{\M}{{\omega}}
\renewcommand{\d}{\ensuremath{\mathrm{d}}}
\newcommand{\R}{\ensuremath{\mathbb{R}}}

\newcommand{\N}{\ensuremath{\mathbb{N}}}
\newcommand{\eps}{\ensuremath{\varepsilon}}
\newcommand{\verti}[1]{\ensuremath{\left\lvert #1 \right\rvert}}
\newcommand{\vertii}[1]{\ensuremath{\left\lVert #1 \right\rVert}}
\newcommand{\vertiii}[1]{{\vert\kern-0.25ex\vert\kern-0.25ex\vert #1
    \vert\kern-0.25ex\vert\kern-0.25ex\vert}}
\renewcommand{\d}{\ensuremath{{\rm d}}}

\newcommand{\bD}{\mathbf{D}}

\newcommand{\ignore}[1]{}

\newtheorem{theorem}{Theorem}[section]
\newtheorem{definition}[theorem]{Definition}
\newtheorem{proposition}[theorem]{Proposition}
\newtheorem{remark}[theorem]{Remark}
\newtheorem{lemma}[theorem]{Lemma}
\newtheorem{corollary}[theorem]{Corollary}

\numberwithin{equation}{section}

\title[Self-similar solutions to thin--film equations]{Existence and regularity of source-type self--similar  solutions for stable thin-film equations}

\author[ M. Majdoub and S. Tayachi]{Mohamed Majdoub and Slim Tayachi}

\address[M. Majdoub]{Department of Mathematics, College of Science, Imam Abdulrahman Bin Faisal University, P. O. Box 1982, Dammam, Saudi Arabia.\newline Basic and Applied Scientific Research Center, Imam Abdulrahman Bin Faisal University, P.O. Box 1982, 31441, Dammam, Saudi Arabia}
\email{\sl mmajdoub@iau.edu.sa}

\address[S. Tayachi]{Universit\'e de Tunis El Manar, Facult\'e des Sciences de Tunis, D\'epartement de math\'ematiques, Laboratoire \'equations aux d\'eriv\'ees partielles (LR03ES04), 2092 Tunis, Tunisie}
\email{\sl slim.tayachi@fst.rnu.tn}

\begin{document}
\begin{abstract}
We investigate the existence and the boundary regularity of source-type self-similar
solutions to the thin-film equation
$h_t=-(h^nh_{zzz})_z+(h^{n+3})_{zz},$ $ t>0,\; z\in \R;\;
h(0,z)= \M \delta(z)$ where $n\in (3/2,3),\; \M > 0$ and $\delta$ is
the Dirac mass at the origin. It is known that the
leading order expansion near the edge of the support coincides with
that of a traveling-wave solution for the standard thin-film equation:
$h_t=-(h^nh_{zzz})_z$. In this paper we sharpen this result,
proving that the higher order corrections are analytic with respect
to three variables: the first one is just the {spatial} variable,
whereas the second and the third (except for $n = 2$) are irrational
powers of it. It is known  that this third variable
does not appear for the thin-film equation without gravity.
\end{abstract}


\date{\today}

\maketitle


\section{Introduction}
In this paper we study the existence and regularity of  source-type self-similar solutions to the thin-film equation:
\begin{subequations}\label{freebdry}
\begin{align}
h_t + (h^n h_{zzz})_z & = (h^{n+3})_{zz}  \quad \mbox{for $t > 0$ and $z \in (Z_-(t), Z_+(t))$},\label{tfe}\\
h = h_z & = 0 \quad \mbox{for $t > 0$ and at $z = Z_{\pm}(t)$},\label{bdry1}\\
\dot Z_{\pm}(t)& = \lim_{z \to Z_{\pm}(t)} h^{n-1}\,h_{zzz}  \quad \mbox{for $t > 0$},\label{bdry2}\\
h(0,z) &= \M\delta(z). \label{bdry3}
\end{align}
\end{subequations}
The function $h = h(t,z)>0$ describes the height or the thickness of a two-dimensional viscous thin-film on a one-dimensional flat solid as a function of time $t > 0$ and the lateral variable $z$. {The parameter $n>0$ stands for the mobility exponent, $\omega>0$ represents the mass and $\delta$ is the Dirac distribution at the origin. Here we are concerned with the range $n\in (3/2,3)$.} The term $(h^{n+3})_{zz}$  represents the effect of the gravity. The plus sign in front of $(h^{n+3})_{zz}$ leads to a stabilizing term (droplet on the ground) as opposed to
the droplet at the ceiling (destabilization). The functions $Z_{\pm}(t)$ define the boundary of the droplet, to which we refer to as contact lines due to their analog for three-dimensional films. Then {conditions}~\eqref{bdry1}$_1$ merely define the contact lines, whereas conditions~\eqref{bdry1}$_2$ state that the contact angle between the liquid-gas and liquid-solid interfaces vanishes (commonly referred to as ``complete wetting regime''). {Conditions}~\eqref{bdry2} {are} of kinematic character. {They state} that the (vertically averaged) velocity of the film $h^{n-1}h_{zzz}$ at the contact lines equals the contact line velocities. One then easily verifies that the mass $\displaystyle\int_{Z_-(t)}^{Z_+(t)} h(t,z) \, \d z$ is a conserved quantity.
 See \cite{NS, ODB1997, OR1992, TK2004} for a survey and more explanations. See also the references \cite{Bertozzi98, BB2001, BP1996, Myers98, ODB1997}.\\

The source-type self-similar solutions of the standard thin-film equation
\begin{equation}
\label{stfe}
h_t + (h^n h_{zzz})_z = 0
\end{equation}
has been studied by many authors, see \cite{GGO} and references therein. In particular, the existence and the asymptotic behavior was established in \cite{BPW}. Uniqueness in the class of even solutions was proved in \cite{BPW} and recently the unconditional uniqueness was obtained in \cite{MMT2018}. The asymptotic behavior given in \cite{BPW} is refined in \cite{GGO}. The result of \cite{BPW} was extended to the thin-film equation with gravity \eqref{tfe} in \cite{B}, but without proving uniqueness and it is shown that the leading term is the same as for \eqref{stfe}. Our aim, as in \cite{GGO}, is to refine the asymptotic behavior obtained in  \cite{B}. Since no uniqueness results are known for even source-type self-similar solutions of \eqref{tfe}, we will not necessarily expand the
solutions obtained in \cite{B}. Precisely, we prove the existence of even source-type self-similar solutions and give their
refined asymptotics. In fact only the first expansion is given in \cite{B}.

{
A slightly more general version of the stable thin-film equation is
given by
\begin{equation*}
\label{FODTF}
h_t + (h^n h_{zzz})_z  = (h^{m})_{zz},
\end{equation*}
where $n, m>0.$ This equation is relevant to surface tension dominated motion of thin viscous films and spreading droplets. The second-order term in the equation, $(h^{m})_{zz}$, arises as a cut off of van der Waals interactions \cite{Gre, BP1994}.}
{In the case $m=n+3$, the last equation enjoys a mass invariant scaling transformation \cite{B, LW}.}

If $h$ is a solution of \eqref{tfe}, then
\[
h_\lambda(t,z)=\lambda h\left(\lambda^{n+4}t,\lambda z\right), \; \lambda >0,
\]
is a solution of \eqref{tfe} on $\Big(\lambda^{-1}\,Z_{-}(\lambda^{n+4}t),\lambda^{-1}\,Z_{+}(\lambda^{n+4}t)\Big)$. Self-similar solutions are such that $h_\lambda\equiv h,$ for all $\lambda>0.$ Taking $\lambda=t^{-{\frac{1}{n+4}}}$ we see that $h$ is a self-similar solution if and only if
\begin{equation}\label{selfsimilar}
h(t,z)= t^{-{\frac{1}{n+4}}}\mathcal{H}\left(t^{-{\frac{1}{n+4}}}z\right),\quad Z_{\pm}(t)=t^{{\frac{1}{n+4}}}\,Z_{\pm}(1),
\end{equation}
where $\mathcal{H}:=h(1,\cdot)$ is the profile of the self-similar solution. We look for regular profiles $\mathcal{H}: \R\longrightarrow [0,\infty)$ that are even  and have compact support $[-a,a],\; a>0,$ $\mathcal{H}>0$ on $(-a,a)$ implying $Z_\pm(t) = \pm a t^{\frac{1}{n+4}}$. Since $\mathcal{H}$ is even, $\mathcal{H}'(0)=0$. By ~\eqref{bdry1} we have
\[
\mathcal{H}=\mathcal{H}'=0\mbox{ at } \pm a.
\]
The conservation of mass, together with ~\eqref{selfsimilar}-\eqref{bdry3}, gives
$$
\int_\R h(t,z)dz =\int_{-a}^a
\mathcal{H}(z)dz=\M.
$$
Since $h$ satisfies equation \eqref{tfe}, $\mathcal{H}$ satisfies equation \eqref{tfeself} below, where we used \eqref{bdry2} after integration.

Hence, we have to look for pairs $\left(a, \mathcal{H}\right)$ solving the problem
\begin{subequations}\label{freebdryself}
\begin{align}
\mathcal{H}^n\mathcal{H}^{'''} & = {\frac{1}{n+4}}y\mathcal{H}+(n+3)\mathcal{H}^{n+2}\mathcal{H}'
  \quad \mbox{for $y \in (-a, a)$},\label{tfeself}\\
&\; \mathcal{H} = \mathcal{H}' = 0 \quad \mbox{at $y= \pm a$},\;  \label{bdryself1}\\
\mathcal{H}(-y)&=\mathcal{H}(y)>0 \quad \mbox{on $(-a,a)$},\label{bdryself2}\\
\int_{-a}^a \mathcal{H}(y)\d y &= \M>0. \label{bdryself3}
\end{align}
\end{subequations}
Clearly \eqref{bdryself2} implies $\mathcal H^\prime(0) = 0$.
Let
\[
\tilde{\mathcal{H}}(y)=(n+4)^{{\frac{1}{n}}}a^{-{\frac{4}{n}}}\mathcal{H}(ay),\; \; y\in (-1,1)
\]
and
\begin{equation}
\label{mu} \mu=(n+3)(n+4)^{-{\frac{2}{n}}}a^{2+{\frac{8}{n}}}.
\end{equation}
Then $\tilde{\mathcal{H}}$ solves the problem
\begin{subequations}\label{freebdryselft}
\begin{align}
\tilde{\mathcal{H}}^n\tilde{\mathcal{H}}^{'''} & =
y\tilde{\mathcal{H}}+\mu\tilde{\mathcal{H}}^{n+2}\tilde{\mathcal{H}}'
  \quad \mbox{for $y \in (-1, 1)$},\label{tfeselft}\\
\tilde{\mathcal{H}}'(0) &= 0,\; \tilde{\mathcal{H}} = \tilde{\mathcal{H}}' = 0 \quad \mbox{at
$y= \pm 1$},\;  \label{bdryselft1}\\
\tilde{\mathcal H}(-y)&=\tilde{\mathcal H}(y)>0 \quad \mbox{on $(-1,1)$},\label{bdryselft2}\\
\int_{-1}^1 \tilde{\mathcal{H}}(y) \d y &= {\sqrt{n+3}\over
\sqrt{\mu}}\M:=\kappa(\mu)>0. \label{bdryselft3}
\end{align}
\end{subequations}
Let us apply the shift $y=-1+x$ and put
$H(x)=\tilde{\mathcal{H}}(-1+x).$ Using the symmetry of
$\tilde{\mathcal{H}},$ the problem reduces to finding a pair $(\mu, H)\in (0,\infty)\times C^1\big([0,1]\big)\cap C^3\big((0,1)\big)$ such that
\begin{subequations}\label{freebdryselfts}
\begin{align}
H^{n-1}H^{'''} & = -1+x+\mu H^{n+1}H',
  \quad \mbox{ $x \in (0, 1]$},\label{tfeselfts}\\
H(0) & = H'(0) = 0,  \label{bdryselfts1}\\
H'(1)& = 0,\label{bdryselfts2}\\
\int_{0}^1 H(y) \d y &={1\over 2}\kappa(\mu)>0. \label{bdryselfts3}
\end{align}
\end{subequations}
As in \cite{GGO}, we denote by
\begin{equation} \label{Htw}
H_{\mathrm{TW}}(x):=A^{-{\nu\over 3}}x^\nu,\,\forall\, x>0,
\end{equation}
a traveling-wave profile to \eqref{stfe}, i.e. a solution of
\begin{subequations}\label{freebdrytw}
\begin{align}
H_{\mathrm{TW}}^{n-1}H_{\mathrm{TW}}^{'''} & = -1,
  \quad \mbox{ $x>0$},\label{tfetw}\\
H_{\mathrm{TW}}(0) & = H_{\mathrm{TW}}'(0) = 0,  \label{bdrytw1}
\end{align}
\end{subequations}
where
\begin{equation}
\label{Anu} \nu:={3\over n},\; A=\nu(\nu-1)(2-\nu) .
\end{equation}
Clearly, $n\in (3/2,3)$ implies $\nu\in (1,2)$ and $A>0$.

 The traveling-wave profile $H_{\mathrm{TW}}$ solves the leading order equation of \eqref{tfeselfts} for $x \ll 1$ (see Lemma \ref{Lexpansions}). Therefore, the solution to \eqref{freebdryselfts} will have the same leading order asymptotic as $x \searrow 0$. The existence of solutions to \eqref{freebdryselfts} which behave like $H_{\mathrm{TW}}$ as $x\searrow 0$ is proved by Beretta. See \cite[Theorem 5.1, p. 760]{B}.

 Our aim is to prove the existence of solutions to \eqref{freebdryselfts} and give a more refined asymptotic than that of \cite{B}.  We now give the main result of this paper.
\begin{theorem}
\label{th1}
Let $3/2<n<3$. Then we have the following:

 (i) There exists $\eps>0$ such that for any $\mu>0$ there exists a solution $H_\mu\in C^1\big([0,1]\big)\cap C^3\big((0,1))$ of \eqref{tfeselfts}-\eqref{bdryselfts1}-\eqref{bdryselfts2} satisfying
\begin{equation}
\label{development}
\hspace{-0.45cm} H_\mu(x)= A^{-{\nu\over
3}}x^\nu\left(1+\bar{u}\left(x,b(\mu)x^\beta,\mu x^\gamma\right)\right),\, 0\leq x\leq \min\left\{\eps^2, {\left({\eps\over b(\mu)}\right)}^{{1\over\beta}}, {\left({\eps^2\over\mu}\right)}^{{1\over\gamma}}\right\}
\end{equation}
for some $b(\mu)>0$, where $\bar{u}(x_1, x_2, x_3) : [0,\eps^2]\times[0,\eps]\times[0,\eps^2]\to \R$ is an analytic function with
\[
\bar{u}(0,0,0)=0,\; \partial_2\bar{u}(0,0,0)<0,
\]
and $\nu$, $A$ are given by \eqref{Anu} and
\begin{equation}
\label{betagamma} \beta:={\sqrt{-3\nu^2+12\nu-8}-3\nu+4\over 2},\; \gamma:=2(1+\nu).
\end{equation}

 (ii) There exists $\bar{\mu}>0$ such that the solution $H_{\bar{\mu}}$ satisfies also \eqref{bdryselfts3}.
\end{theorem}
{The previous theorem proves that the higher order corrections are analytic with respect to three variables: the first one is just the {spatial} variable,
whereas the second and the third (except for $n = 2$) are irrational
powers of it. It is known  that this third variable
does not appear for the thin-film equation without gravity \cite{GGO}. This shows the impact of the gravity on the regularity. See \eqref{DASYMP} below where we show that  $\partial_1\bar{u}(0,0,0)>0$ and $\partial_3\bar{u}(0,0,0)>0$.} {The fractional power $\beta$ is obtained when linearizing \eqref{tfeselfts} around the traveling-wave $H_{TW}$ and appears as a root of a polynomial $p$. See Section 2 below.}

The proof of Theorem \ref{th1} mainly uses the method introduced in \cite{GGO} and some ideas in \cite{B}. We first construct a local solution to \eqref{tfeselfts}-\eqref{bdryselfts1} which is analytic in the three variables $x,\; b\,x^\beta$ and $\mu\,x^\gamma$. To do this, we unfold the singular behavior and construct a local solution for the resulting nonlinear partial differential equation. See \eqref{41}-\eqref{41_bc} and Proposition \ref{p2} below. Our approach in the rest of the proof is based on a shooting argument with respect to two parameters. We shoot with respect to the parameter $b>0$ to fulfill the boundary condition \eqref{bdryselfts2}. See Proposition \ref{P5.01} below. Finally, by shooting with respect to the parameter $\mu>0$ we fulfill the mass condition \eqref{bdryselfts3}. See Proposition \ref{barmu} below.

 Let us mention that in \cite{GGO}, the shooting argument is done only with respect to one parameter. Indeed, by scaling argument we can reach any mass $\omega>0$ from any given solution of  \eqref{tfeselfts}-\eqref{bdryselfts1}-\eqref{bdryselfts2} without gravity which is not possible in our case. This justify why we need Part (ii) in the previous theorem.

\begin{remark}
{\rm Our arguments are also valid to construct  local regular solutions of \eqref{tfeselfts}-\eqref{bdryselfts1} with $\mu$ replaced by $-\mu,$  which holds when studying source-type self-similar solutions for the unstable thin-film equation   (this is the droplet at the ceiling)
\[
h_t=-(h^nh_{zzz})_z-(h^{n+3})_{zz}.
\]
See the proof of Proposition \ref{p2} below. To construct regular solutions satisfying the hole of the problem \eqref{freebdryselfts}, we think that the mass $\M$ should be less then the critical mass $\M_c=2\pi \sqrt{2/3}.$ See  \cite[p. 237]{WBB},   \cite[reference 6, p. 254]{WBB}, \cite[footnote, p.1711]{SP} and \cite{LW} for this restriction on $\M$. See also \cite{BP1998, BP2000, EGK2007I, EGK2007II, NS, S2009, SP, WBB} for the unstable thin-film equation. We mention also that self-similar solutions to stable thin-film equation related to \eqref{tfe} are done in \cite{EG2011I,EG2011II}.}
\end{remark}

The rest of this paper is devoted to the proof of the main result, that is Theorem \ref{th1}.
Section 2 deals with the unfolding of the singularity in the three variables $x,\, b\,x^\beta$ and  $\mu\,x^\gamma$. Section 3 is devoted to the study of the related linear problem. In the forth section, we prove the local existence for the nonlinear problem. Section 5 is devoted to the shooting arguments in order to obtain the desired existence and regularity. In the sequel, $C$ will
be used to denote a constant which may vary from line to line.
We also use $A\lesssim B$ to denote an estimate of the form $A\leq C B$
for some absolute constant $C$, $A\approx B$ if $A\lesssim B$ and $B\lesssim A$ and $A\ll B$ if $A$ is sufficiently small with respect to $B$. Finally, we use the notation $\N_0=\N\cup\{0\}=\{0,1,2,\cdots\}.$

\section{Unfolding of the singularity}
As in \cite{GGO}, we factor off the leading order behavior $H_{\mathrm{TW}} = A^{-{\frac \nu 3}}x^\nu$, i.e.
\begin{equation}
\label{30} H(x)=:A^{-{\nu \over 3}}x^\nu F(x).
\end{equation}
{Motivated by \cite{B}, we impose}
\begin{equation}
\label{31} F(0)=1.
\end{equation}
Equation \eqref{tfeselfts} becomes
\begin{equation}
\label{32} F^{n-1}q(D)F=A(-1+x)+\mu A^{-{2\nu \over
3}}x^{2\nu+2}F^{n+1}(D+\nu)F,
\end{equation}
where{, as in \cite{GGO}, the scaling-invariant logarithmic derivative operator $D$ is defined by}
\begin{equation}
\label{33} D:=x\partial_x=\frac{d}{ds}, \;\; s:=\ln x,
\end{equation}
and the polynomial $q$ is given by
\begin{equation}
\label{35} q(\xi)=(\xi+\nu)(\xi+\nu-1)(\xi+\nu-2).
\end{equation}

Put
\[
F(x)=:1+u(x).
\]
Then since $q(D)1=-A,$ we have,
\begin{eqnarray*}
F^{n-1}q(D)F &=& (1+u)^{n-1}q(D)\big(1+u\big)\\
&=&-A(1+u)^{n-1}+(1+u)^{n-1}q(D)u \\
&=&-A-A[(1+u)^{n-1}-1]+(1+u)^{n-1}q(D)u \\
&=&-A-A\big[(1+u)^{n-1}-1-(n-1)u\big]+\big[(1+u)^{n-1}-1\big]q(D)u\\&&+q(D)u-(n-1)Au\\
&=&-A-A\big[(1+u)^{n-1}-1-(n-1)u\big]+\big[(1+u)^{n-1}-1\big]q(D)u\\&&+p(D)u,
\end{eqnarray*}
where
\begin{equation*}
\label{36} p(D)u=q(D)u-(n-1)Au.
\end{equation*}
Hence, using \eqref{35}, the polynomial $p(\xi)$ is given by
\begin{eqnarray*}
\label{37}
p(\xi)&=&\xi^3+3(\nu-1)\xi^2+(3\nu^2-6\nu+2)\xi-3(\nu-1)(2-\nu)\\
&=& \label{38} (\xi+1)(\xi-\alpha)(\xi-\beta)
\end{eqnarray*}
where $\beta$ is given by \eqref{betagamma} and $\alpha$ is given by
\begin{equation*}
\label{36bis}\alpha := {{-\sqrt{-3\nu^2+12\nu-8}-3\nu+4\over 2}}.
\end{equation*}
Clearly since $n\in (3/2,3),$ then $\alpha\in (-2,0)$ and $\beta\in
(0,1).$

Problem \eqref{31}--\eqref{32} becomes
\begin{eqnarray}
\nonumber
p(D)u&=&Ax+A\big[(1+u)^{n-1}-1-(n-1)u\big]\\ \nonumber &&-\big[(1+u)^{n-1}-1\big]q(D)u\\
\label{41} &&+\mu A^{-{2\nu \over 3}}x^\gamma(1+u)^{n+1}(D+\nu)(1+u),\; x\in (0,1]\\
u(0)&=&0. \label{41_bc}
\end{eqnarray}
We will study the corresponding linear problem
\begin{eqnarray}
\label{39}
p(D)u &= &f,\; x\in (0,1] \\ \label{39bc}
u(0) &=& 0.
\end{eqnarray}
For that purpose, we introduce a second and third variable
\[
y:=bx^\beta,\; z:=\mu x^{\gamma}
\]
for some $b\in \R,\; \mu>0$ to be fixed later.
Let us explain the reason for that.

One cannot expect the solution $u(x)$ of \eqref{41} to be smooth in the single variable $x$, since this, together with boundary condition \eqref{41_bc}, rules out all homogeneous solutions $x^{-1},$ $x^\alpha$, and $x^\beta$ to the corresponding linear problem \eqref{39}. Of these, the only one that is compatible with boundary condition~\eqref{39bc} is the solution $x^\beta$. Note, however, that $\frac{\d^k}{\d x^k} x^\beta$ is singular in $x=0$ for $k \ge 1$ and so there can only be one solution $u(x)$ to \eqref{39} that is smooth with respect to the single variable $x$ for smooth right-hand sides $f(x)$. Hence one introduces the artificial variable $y :=  b x^\beta$, being the only solution of \eqref{39} with $f \equiv 0$ that obeys \eqref{39bc}.

One cannot expect the solution $u(x)$ to be a smooth function in the two variables $x$ and $x^\beta$, since the right hand side of equation~\eqref{41} is, for $n \ne 2$, not smooth in the two variables $x$ and $y = b x^\beta$. This is why one introduces the artificial variable $z :=  \mu x^\gamma$.

If $v(x)$ and $\bar v(x,bx^\beta,\mu x^\gamma)$ are regular
functions related via $v(x)=\bar v(x,bx^\beta,\mu x^\gamma)$
we have {by \eqref{33}} $Dv(x)=\bar\bD\bar v(x,bx^\beta,\mu x^\gamma),$
where
\begin{equation}
\label{barD}
\bar{\bD}:=x\partial_x+\beta y \partial_y+\gamma z\partial_z.
\end{equation}
In order to unfold the singular behavior, we introduce  also
\[
u(x)=\bar u(x,bx^\beta, \mu x^\gamma).
\]
 {Using the identification between $D$ and $\bar{\bD}$,} the conditions $u(0)=0$ and $u(x)\sim -bx^\beta$ as $x \searrow 0$ combined with equation \eqref{39} translate to the linear problem
\begin{eqnarray}
\label{40bis}
p(\bar{\bD})\bar{u}&=&\bar{f},\;\mbox{ for } x>0,\; y>0,\; z>0,\\ \label{40tris}
\big(\bar{u},\partial_y\bar{u}\big)(0,0,0)&=&\big(0,-1\big).
\end{eqnarray}
In fact equation \eqref{41} reads in the new variables
\begin{eqnarray*}
p(\bar{\bD})\bar{u}&=&Ax+A\big[(1+\bar{u})^{n-1}-1-(n-1)\bar{u}\big]\\
  &&-\big[(1+\bar{u})^{n-1}-1\big]q(\bar{\bD})\bar{u}\\
 &&+A^{-{2 \over 3}\nu}z(1+\bar{u})^{n+1}(\bar{\bD}+\nu)(1+\bar{u}).
\end{eqnarray*}
Then the solution $\bar u(x,y,z)$ of \eqref{40bis}-\eqref{40tris}  coincides with that of \eqref{41}-\eqref{41_bc} in the case $y = \overline{b} x^\beta$, $z = \overline{\mu}x^\gamma$, for fixed values $(\overline{b}, \overline{\mu})$, chosen such that condition~\eqref{bdryselfts2} as well as condition~\eqref{bdryselfts3} are fulfilled. The freedom to choose two real parameters $b$ and $\mu$ will play a crucial role to fulfill two additional conditions.

%
\section{Well-posedness for the linear problem}
We introduce the notation $(x,y,z)=:(x_1,x_2,x_3)$, as well as $\partial_{x_i}:=\partial_i$ for $i=1,\, 2, \, 3$. Let us set
\[
\bar{u}=:\bar{u}_0-x_2.
\]
We will construct a solution to the linear problem with
homogeneous boundary condition:
\begin{eqnarray}
\label{41new}
p(\bar{\bD})\bar{u}_0&=&\bar{f},\;\mbox{ for } x_1>0,\; x_2>0,\; x_3>0,\\
\label{42}
\big(\bar{u}_0,\partial_2\bar{u}_0\big)(0,0,0)&=&\big(0,0\big).
\end{eqnarray}
A key tool in our construction will be the following lemma.
\begin{lemma}
\label{lem1}
Let $\Lambda\leq\beta$, and consider the problem
\begin{subequations}\label{Linear}
\begin{align}
\left(\bar{\bD}-\Lambda\right)\bar{u}&=\bar{f}, \label{Linear1}\\
\big(\bar{u},\partial_2\bar{u}\big)(0,0,0)&=\big(0,0\big).\label{Linear2}
\end{align}
\end{subequations}
Then, for all smooth function $\bar{f}(x_1,x_2,x_3)$ with $\big(\bar{f},\partial_2\bar{f}\big)(0,0,0)=\big(0,0\big)$, the function
\bq
\label{T}
\bar{u}(x_1,x_2,x_3)=\left(T_{\Lambda}\,\bar{f}\right)(x_1,x_2,x_3):=\int_0^1\,r^{-\Lambda}\,\bar{f}(r x_1,r^\beta x_2, r^\gamma x_3)\,\frac{dr}{r},
\eq
is a smooth solution of \eqref{Linear} such that
\bq
\label{Estim1}
\sum_{j=0}^1\vertii{\partial_1^k\partial_2^\ell\partial_3^m\bar{\bD}^j\bar{u}} \lesssim \vertii{\partial_1^k\partial_2^\ell\partial_3^m\bar{f}},\quad (k,\ell,m)\in\N_0^3,
\eq
with $(k,\ell,m)\not\in \{\,(0,0,0),\; (0,1,0)\}$ if $\Lambda=\beta$, where $\vertii{\cdot}$ denotes the sup-norm on an arbitrary cuboid $[0,\ell_1]\times [0,\ell_2]\times [0,\ell_3]$. Moreover, we have the commutation property
\[
T_\Lambda\bar{\bD}=\bar{\bD} T_\Lambda\,.
\]
\end{lemma}
\begin{proof}
Although the proof is similar to that in \cite[Lemma 1]{GGO}, we give it for completeness. Let us first show that formula \eqref{T} defines a smooth function satisfying \eqref{Linear2}. Writing
\[
\verti{\left(T_{\Lambda}\,\bar{f}\right)(x_1,x_2,x_3)}\leq \vertii{\bar{f}}\,\int_0^1\,\frac{dr}{r^{1+\Lambda}},
\]
we see that $\bar{u}$ is well defined for $\Lambda <0$. Assume now that $0\leq \Lambda\leq \beta$ and expanding
\[
\bar{f}(r x_1,r^\beta x_2, r^\gamma x_3)=r x_1\partial_1\bar{f}(0,0,0)+r^\gamma x_3\partial_3\bar{f}(0,0,0)+O_{(x_1,x_2,x_3)}(r^{2\beta}),
\]
we end up with
\beqn
\verti{\left(T_{\Lambda}\,\bar{f}\right)(x_1,x_2,x_3)}&\leq& \vertii{x_1\partial_1\bar{f}(0,0,0)}\int_0^1\,\frac{dr}{r^\Lambda}+
\vertii{x_3\partial_3\bar{f}(0,0,0)}\int_0^1\,\frac{dr}{r^{\Lambda-\gamma}}\\&+&
C(x_1,x_2,x_3)\,\int_0^1\,\frac{dr}{r^{\Lambda-2\beta}}<\infty.
\eeqn
The fact that $\partial_2\bar{f}(0,0,0)=0$ implies
$$
\partial_2\bar{f}(r x_1,r^\beta x_2, r^\gamma x_3)=O\,(r x_1+r^\beta x_2+r^\gamma x_3)\,.
$$
It follows that
\begin{equation}
\label{partialtwo}
\partial_2\bar{u}(x_1,x_2,x_3)=
\int_0^1\,r^{-\Lambda+\beta}\;\partial_2\bar{f}(r x_1,r^\beta x_2, r^\gamma x_3)\frac{dr}{r},
\end{equation}
is well defined. The boundary conditions \eqref{Linear2} follows from \eqref{T} and \eqref{partialtwo}.
To prove the smoothness, observe that
\[
\partial_1^k\partial_2^\ell\partial_3^m\bar{u}(x_1,x_2,x_3)=
\int_0^1\,r^{-\Lambda+k+\ell\beta+m\gamma}\;\partial_1^k\partial_2^\ell\partial_3^m\bar{f}(r x_1,r^\beta x_2, r^\gamma x_3)\frac{dr}{r}\,.
\]
Since $\Lambda\leq\beta$ and the cases $(k,\ell,m)\in \{\,(0,0,0),\; (0,1,0)\}$ was handled, we conclude that the integral converges and the derivatives up to any order are well defined.

Recalling that
\[
\bar{\bD}=x_1\partial_1+\beta x_2\partial_2+\gamma x_3\partial_3,
\]
we compute
\beq
\label{LinEq}
\nonumber
\bar{\bD}\bar{u}&=&\int_0^1\,r^{-\Lam}\Big[r x_1\partial_1\bar{f} +\beta r^\beta x_2\partial_2\bar{f}+\gamma r^\gamma x_3\partial_3\bar{f}\Big](r x_1,r^\beta x_2, r^\gamma x_3)\frac{dr}{r}\\
&=&\int_0^1\,r^{-\Lam}\frac{d}{dr}\Big[\bar{f}(r x_1,r^\beta x_2, r^\gamma x_3)\Big]\,dr\\
\nonumber
&=&\Big[r^{-\Lam}\bar{f}(r x_1,r^\beta x_2, r^\gamma x_3)\Big]_0^1+\Lam\int_0^1\,\,r^{-\Lam}\,\bar{f}(r x_1,r^\beta x_2, r^\gamma x_3)\,\frac{dr}{r}\\
\nonumber
&=&\bar{f}(x_1,x_2,x_3)+\Lam \bar{u}(x_1,x_2,x_3),
\eeq
where we have used the fact that $\bar{f}(0,0,0)=\partial_2\bar{f}(0,0,0)=0$ to deduce
\[
\left. r^{-\Lam}\bar{f}(r x_1,r^\beta x_2, r^\gamma x_3)\right|_{r=0}=0\,.
\]
We have by definition of $T_\Lam$:
\[
T_\Lam\bar{\bD}\bar{f}=\int_0^1\,r^{-\Lam}\Big[r x_1\partial_1\bar{f} +\beta r^\beta x_2\partial_2\bar{f}+\gamma r^\gamma x_3\partial_3\bar{f}\Big](r x_1,r^\beta x_2, r^\gamma x_3)\frac{dr}{r},
\]
which is equal to $\bar{\bD}\bar{u}=\bar{\bD}T_\Lam\bar{f}$ thanks to \eqref{LinEq}. Finally, \eqref{Linear2} is obvious and estimate \eqref{Estim1} follows from the equation
\[
\bar{\bD}\bar{u}=\Lam\bar{u}+\bar{f},
\]
and the fact that
\[
\partial_1^k\partial_2^\ell\partial_3^m\bar{u}=
T_{\Lam-k-\ell\beta-m\gamma}\,\partial_1^k\partial_2^\ell\partial_3^m\bar{f}\,.
\]
\end{proof}
A straightforward consequence of Lemma~\ref{lem1} is:
\begin{proposition}
\label{p1} There exists a linear operator $T$ such that for all smooth functions $\bar{f}(x_1,x_2,x_3)$ with $(\bar{f},\partial_2\bar{f})(0,0,0)=(0,0),$ the function $\bar{u}(x_1,x_2,x_3):=(T\bar{f})(x_1,x_2,x_3)$ is the unique smooth solution of  \eqref{41new}--\eqref{42}. Furthermore, $\bar{u}(x_1,x_2,x_3)$ satisfies the estimates
\begin{equation}
\label{40}
\sum_{j=0}^3\vertii{\partial_1^k\partial_2^l\partial_3^m \bar{\bD}^j
\bar{u}} \lesssim \vertii{\partial_1^k\partial_2^l\partial_3^m
\bar{f}} ,\; \forall\;(k,l,m)\in \N_0^3\setminus \{(0,0,0),\;
(0,1,0)\}.
\end{equation}
\end{proposition}
\begin{proof}
As in \cite{GGO}, we set
\[
T:=T_\beta\,T_{-1}\,T_{\alpha}\,.
\]
Hence, $\bar{u}:=T\bar{f}$ is well defined, smooth and satisfies the problem {\eqref{41new}-\eqref{42}}. The estimate \eqref{40} follows from Lemma \ref{lem1}. The uniqueness follows from \eqref{40} and Part a) of Lemma \ref{L31} below.
\end{proof}

\section{Local existence}
The unfolded function $\bar{u}(x, bx^\beta, \mu x^\gamma)$ (with $u(x) = \bar u(x,\overline{b} x^\beta, \overline{\mu} x^\gamma)$) shall satisfy the following boundary value problem
\begin{subequations}\label{Nonlinear}
\begin{align}
p(\bar{\bD})\bar{u}&=\bar{f}_{\bar{u}},\quad \mbox{for}\;\; x_1,x_2,x_3>0, \label{Nonlinear1}\\
\big(\bar{u},\partial_2\bar{u}\big)(0,0,0)&=\big(0,-1\big),\label{Nonlinear2}
\end{align}
\end{subequations}
where
\beq
\nonumber
\indent\indent \indent\indent \indent\indent \bar{f}_{\bar{u}}&=&A x_1-\left((1+\bar{u})^{n-1}-1\right)q(\bar{\bD})\bar{u}\\&&\label{source} +
A\left[(1+\bar{u})^{n-1}-1-(n-1)\bar{u}\right]\\ \nonumber &&+A^{-\frac{2}{3}\nu}x_3(1+\bar{u})^{n+1}\left(\bar{\bD}+\nu\right)(1+\bar{u})\,.
\eeq
The main result of this section is the following.
\begin{proposition}
\label{p2}
There exist $\eps\in(0,1)$ and $\bar{u}(x_1,x_2,x_3)$ analytic in
$[0,\eps^2]\times[0,\eps]\times[0,\eps^2]:=Q_\eps$ such that $\bar{u}$ solves \eqref{Nonlinear} in $Q_\eps$.
\end{proposition}
The proof uses a fixed-point argument. In order to establish the contraction property, we need the following lemma:
\begin{lemma}
\label{L31}
 Let $\bar{f}(x_1,x_2,x_3)$, $\bar{g}(x_1,x_2,x_3)$ be smooth. Then we have
\begin{itemize}
\item[a)] if $\big(\bar{f},\partial_2\bar{f}\big)(0,0,0)=\big(0,0\big)$, then
\[
\vertii{\bar{f}}+\eps\vertii{\partial_2\bar{f}}\lesssim \eps^2\left(\vertii{\partial_1\bar{f}}+\vertii{\partial^2_2\bar{f}}+\vertii{\partial_3\bar{f}}\right).
\]
\item[b)] $\verti{\bar{f}\bar{g}}_0\leq \verti{\bar{f}}_0\verti{\bar{g}}_0$, where, for $K, L, M\in\N$,
\[
\verti{\bar{h}}_0=\sum_{k=0}^K\sum_{\ell=0}^L\sum_{m=0}^M\frac{\eps^{2k+\ell+2m}}{k!\ell !m!}\vertii{\partial_1^k\partial_2^\ell\partial_3^m\bar{h}},
\]
where $\vertii{\cdot}$ denotes the sup-norm on $Q_\varepsilon$.
\end{itemize}
\end{lemma}
\begin{proof}
Part a) of the lemma follows immediately from the following representations:
\beqn
\bar{f}(x_1,x_2,x_3)&=&\int_0^{x_2}\int_0^s\partial^2_2\bar{f}(0,\tau,0)d\tau ds
+\int_0^{x_1}\partial_1\bar{f}(s,x_2,0)ds\\&&+\int_0^{x_3}\partial_3\bar{f}(x_1,x_2,s)ds,
\eeqn
and
\beqn
\eps\partial_2\bar{f}(x_1,x_2,x_3)&=&\bar{f}(x_1,\eps,x_3)-\bar{f}(x_1,0,x_3)+
\int_0^{x_2}\tau\partial^2_2\bar{f}(x_1,\tau,x_3)d\tau\\&&-\int_{x_2}^{\eps}(\eps-\tau)
\partial^2_2\bar{f}(x_1,\tau,x_3)d\tau.
\eeqn
We now turn to the proof of Part b). By {Leibniz' rule}, we have
\[
\partial_1^k\partial_2^\ell\partial_3^m\left(\bar{f}\bar{g}\right)=\sum_{k'=0}^k\sum_{\ell'=0}^\ell\sum_{m'=0}^m
\frac{k!\ell!m![\partial_1^{k'}\partial_2^{\ell'}\partial_3^{m'}\bar{f}]
[\partial_1^{k-k'}\partial_2^{\ell-\ell'}\partial_3^{m-m'}\bar{g}]}{(k-k')!(\ell-\ell')!(m-m')!k'!\ell'!m'!}.
\]
Using the fact that $\vertii{u v}\leq \vertii{u}\vertii{v}$, we deduce
\[
\verti{\bar{f}\bar{g}}_0\leq \sum_{k=0}^K\sum_{\ell=0}^L\sum_{m=0}^M \sum_{k'=0}^K\sum_{\ell'=0}^L\sum_{m'=0}^M
a_{k',\ell',m'}\,b_{k-k',\ell-\ell',m-m'},
\]
where
\beqn
a_{k',\ell',m'}&=&\frac{\eps^{2k'+\ell'+2m'}}{k'!\ell' !m'!}\vertii{\partial_1^{k'}\partial_2^{\ell'}\partial_3^{m'}\bar{f}},\\\\
b_{k-k',\ell-\ell',m-m'}&=&\frac{\eps^{2(k-k')+(\ell-\ell')+2(m-m')}}{(k-k')!(\ell-\ell') !(m-m')!}\vertii{\partial_1^{k-k'}\partial_2^{\ell-\ell'}\partial_3^{m-m'}\bar{g}}.
\eeqn
Hence
\[
\verti{\bar{f}\bar{g}}_0\leq \left(\sum_{k=0}^K\sum_{\ell=0}^L\sum_{m=0}^M a_{k,\ell,m}\right)\left(\sum_{k=0}^K\sum_{\ell=0}^L\sum_{m=0}^M b_{k,\ell,m}\right).
\]
This concludes the proof of the lemma.
\end{proof}
We will need the following result for the fixed-point argument:
\begin{lemma}
\label{L32}
 Let $\bar{f}$ be  a smooth function satisfying  $\big(\bar{f},\partial_2\bar{f}\big)(0,0,0)=\big(0,0\big).$ Let $\bar{u}=T\bar{f}$ be the solution of
\begin{eqnarray*}
p(\bar{\bD})\bar{u}&=\bar{f}, \\
\big(\bar{u},\partial_2\bar{u}\big)(0,0,0)&=\big(0,0\big),
\end{eqnarray*}
given by Proposition \ref{p1}. Then we have
\[
\verti{\bar{u}}_1=\verti{T\bar{f}}_1\lesssim \verti{\bar{f}}_0,
\]
where {$|\cdot|_0$ is as in Lemma \ref{L31} and} $|\cdot|_1$ is defined by
$$
|\bar{h}|_1=
\sum_{j=0}^3\,|\bar{\bD}^j\,\bar{h}|_0,
$$
{and both of them are restricted to $Q_\eps$.}
\end{lemma}
\begin{proof}
Since $\left(\bar{\bD}^j\bar{u},\partial_2\bar{\bD}^j\bar{u}\right)(0,0,0)=\big(0,0\big)$, we obtain by part a) of Lemma \ref{L31} and Proposition \ref{p1}
\begin{eqnarray*}
\verti{\bar{u}}_1&\lesssim&\sum_{(k,\ell,m)\not\in\{(0,0,0), (0,1,0)\}}\, \frac{\eps^{2k+\ell+2m}}{k!\ell !m!}\left(\sum_{j=0}^3\,\vertii{\partial_1^k\partial_2^\ell\partial_3^m\bar{\bD}^j\bar{u}}\right)\\\\
&\lesssim&\sum_{k=0}^K\sum_{\ell=0}^L\sum_{m=0}^M \frac{\eps^{2k+\ell+2m}}{k!\ell !m!}\vertii{\partial_1^k\partial_2^\ell\partial_3^m\bar{f}}=\verti{\bar{f}}_0
\end{eqnarray*}
\end{proof}
We now turn to the proof of Proposition \ref{p2}.
\begin{proof}[Proof of Proposition \ref{p2}]
We write $\bar{u}(x_1,x_2,x_3)=:-x_2+\bar{u_0}(x_1,x_2,x_3)$, and rewrite \eqref{Nonlinear1}-\eqref{Nonlinear2} in the equivalent formulation:
\begin{subequations}\label{Nonlinear'}
\begin{align}
p(\bar{\bD})\bar{u_0}&=\bar{f}_{\bar{u}},\quad \mbox{for}\;\; x_1,x_1,x_3>0, \label{Nonlinear'1}\\
\big(\bar{u_0},\partial_2\bar{u_0}\big)(0,0,0)&=\big(0,0\big),\label{Nonlinear'2}
\end{align}
\end{subequations}
where $\bar{f}_{\bar{u}}$ is given by \eqref{source}.  For $K, L, M$ fixed integers, let

$$
{\mathbf S}_{K,L,M}:=\Big\{\,\bar{v}\in C^{K+L+M+3}(Q_\eps);\;\;\;(\bar{v},\partial_2\bar{v})(0,0,0)=(0,0)\;\;\mbox{and}\;\;
|\bar{v}{\tilde|}_1\leq\eps\Big\},
$$
where
\begin{equation}
\label{tilde1}
|\bar{v}{\tilde|}_1:=
|\bar{v}|_1+\sum_{\alpha=(\alpha_1,\alpha_2,\alpha_3)\atop
|\alpha|=K+L+M+3}\;
\frac{\eps^{2\alpha_1+\alpha_2+2\alpha_3}}{\alpha_1!\alpha_2!\alpha_3!}\|\partial^\alpha\bar{v}\|\,.
\end{equation}
Since $\left(C^{K+L+M+3}(Q_\eps), |\cdot{\tilde|}_1\right)$ is a Banach space, it follows that ${\mathbf S}_{K,L,M}$ is a complete metric space as it is closed in $C^{K+L+M+3}(Q_\eps)$. Note that if $\bar{u}\in {\mathbf S}_{K,L,M}$ then (see \cite{GGO})
\[
(\bar{f}_{\bar{u}},\partial_2\bar{f}_{\bar{u}})(0,0,0)=(0,0)\,.
\]
Hence the operator $T$ given by Proposition \ref{p1} is well defined, and we obtain a fixed point equation :
\[
\bar{u}=-x_2+T\bar{f}_{\bar{u}}:={\mathcal T}(\bar{u})\,.
\]
To conclude we will show that ${\mathcal T}$ is a contraction from ${\mathbf S}_{K,L,M}$ into itself. Therefore one has to prove the estimates
\begin{subequations}\label{contraction}
\begin{equation}
|\mathcal T (\bar u) - \mathcal T (\bar v){\tilde|}_1 \lesssim \eps |\bar u - \bar v{\tilde|}_1 \quad \mbox{for all $\bar u, \bar v \in {\mathbf S}_{K,L,M}$}
\end{equation}
and
\begin{equation}
|\mathcal T (\bar u){\tilde|}_1 \lesssim \eps^2 \quad \mbox{for all $\bar u \in {\mathbf S}_{K,L,M}$}
\end{equation}
\end{subequations}
and then choose $\eps > 0$ sufficiently small.
The proof is the same as in \cite{GGO} by using Lemmas \ref{L31}-\ref{L32} and \cite[Lemma 4]{GGO} (for smooth functions with three variables), namely we have.

\begin{lemma}
\label{equilemma4GGO}
 Let $\bar{f}(x_1,x_2,x_3)$, $\bar{g}(x_1,x_2,x_3)$ be smooth functions with $|\bar f{\tilde|}_0,\; |\bar g{\tilde|}_0\leq 1/2.$ Then we have, for any $m\in\R:$

\begin{eqnarray*}
\big|(1+\bar f)^m-1{\tilde{\big|}}_0&\lesssim_m&  |\bar f{\tilde|}_0,\\
|(1+\bar f)^m-(1+\bar g)^m{\tilde|}_0&\lesssim_m& |\bar f-\bar g{\tilde|}_0,\\
|(1+\bar f)^m-m\bar f-(1+\bar g)^m+m\bar g{\tilde|}_0&\lesssim_m&  \max\{|\bar f{\tilde|}_0,\; |\bar g{\tilde|}_0\}|\bar f-\bar g{\tilde|}_0,
\end{eqnarray*}
where
\begin{equation*}
\label{tilde0}
|\bar{v}{\tilde|}_0:=
|\bar{v}|_0+\sum_{\alpha=(\alpha_1,\alpha_2,\alpha_3)\atop
|\alpha|=K+L+M+3}\;
\frac{\eps^{2\alpha_1+\alpha_2+2\alpha_3}}{\alpha_1!\alpha_2!\alpha_3!}\|\partial^\alpha\bar{v}\|,
\end{equation*}
{and $\vertii{\cdot}$ denotes the sup-norm on $Q_\varepsilon$.}
\end{lemma}
The proof of this lemma is the same as in \cite{GGO}, i.e. that it uses the series expansion of the fractional power and the sub-multiplicativity of the norm $|\cdot\tilde{|}_0$. We use Lemma~\ref{L32} to conclude $|\mathcal T (\bar u) - \mathcal T (\bar v)\tilde{|}_1 \lesssim |\bar f_{\bar u} - \bar f_{\bar v}\tilde{|}_0$ as well as $|\mathcal T(\bar u)\tilde{|}_1 \lesssim |\bar f_{\bar u}\tilde{|}_0$ and that \eqref{contraction} can now be established by using Lemma \ref{equilemma4GGO}.  This has been mainly done in \cite{GGO} and we only treat the additional appearing terms in $\bar{f}_{\bar{u}}$. We have
\begin{eqnarray*}
|A^{-\frac{2}{3}\nu}x_3(1+\bar{u})^{n+1}\left(\bar{\bD}+\nu\right)(1+\bar{u})\tilde{|}_0&\lesssim &
|x_3\tilde{|}_0|(1+\bar{u})^{n+1}\tilde{|}_0\left(|\bar{\bD}\bar{u}\tilde{|}_0+|\bar{u}\tilde{|}_0+1\right)\\&\lesssim\eps^2,
\end{eqnarray*}
and
\begin{eqnarray*}
&|A^{-\frac{2}{3}\nu}x_3\left[(1+\bar{u})^{n+1}\left(\bar{\bD}+\nu\right)(1+\bar{u})-(1+\bar{v})^{n+1}\left(\bar{\bD}+
\nu\right)(1+\bar{v})\right]\tilde{|}_0\lesssim &\\&\eps^2
\Big[|(1+\bar{u})^{n+1}\left(\bar{\bD}+\nu\right)(\bar{u}-\bar{v})\tilde{|}_0+
|\big((1+\bar{u})^{n+1}-(1+\bar{v})^{n+1}\big)\left(\bar{\bD}+\nu\right)(1+\bar{v})\tilde{|}_0\Big]\lesssim &\\ &\eps^2\left[|(1+\bar{u})^{n+1}\tilde{|}_0|\bar{u}-\bar{v}\tilde{|}_1+|(1+\bar{u})^{n+1}-(1+\bar{v})^{n+1}\tilde{|}_0\right]
\lesssim &\\ &\eps^2 |\bar{u}-\bar{v}\tilde{|}_1\,.&
\end{eqnarray*}
Since the sets ${\mathbf S}_{K,L,M}$ are nested as $K, L, M$ increase, the fixed point $\bar{u}_0$ is $C^\infty$ and the Taylor series
\[
\sum_{k=0}^\infty\sum_{\ell=0}^\infty\sum_{m=0}^\infty\frac{\partial_1^k\partial_2^\ell\partial_3^m\bar{u_0}(0,0,0)}{k!\ell!m!} x_1^k x_2^\ell x_3^m,
\]
converges absolutely in $Q_\eps$. Moreover the corresponding error terms converge uniformly to zero, then the Taylor series also represents the solution, i.e. the solution is analytic. This concludes the proof of Proposition \ref{p2}.
\end{proof}
\begin{remark}{\rm The result of Proposition \ref{p2} still valid if we replace
$Q_\eps$ by
$\tilde{Q}_\eps=[-\eps^2,\eps^2]\times[-\eps,\eps]\times[-\eps^2,\eps^2]$ for $(b,\mu) \in \R^2$.}
\end{remark}

\section{Regularity}
In this section we give the proof of Theorem \ref{th1}. Until now, we have constructed a solution of (\ref{tfeselfts}) and (\ref{bdryselfts1})
\begin{equation}
\label{e5.01}
H_{b,\mu}(x)=A^{-{\nu\over 3}}x^\nu(1+u_{b,\mu}(x)),
\end{equation}
with
$$\nu={3\over n},\; A=\nu(\nu-1)(2-\nu),\; \mbox{ and  } \; u_{b,\mu}(x)=\bar{u}\left(x, bx^{\beta}, \mu x^\gamma\right),$$
where
$\bar u (x_1,x_2, x_3)$  is given by Proposition \ref{p2}. In particular, $\bar u (x_1,x_2, x_3)$ is analytic in $Q_\eps=[0,\eps^2]\times[0,\eps]\times[0,\eps^2].$ Then $u_{b,\mu},$ hence $H_{b,\mu}$ are defined for
\bq
\label{defhatx}
0\leq x\leq \hat{x}_{b,\mu}(\eps):=\min\left\{\eps^2, \left({\eps\over b}\right)^{{1\over \beta}},\left({\eps^2\over \mu}\right)^{{1\over \gamma}}\right\}.
\eq

We first give the following existence and uniqueness result.
\begin{lemma}\label{ODE-max}
Consider the initial value problem
\begin{subequations}\label{ODE}
\begin{align}
U'''&=(x-1)U^{1-n}+\mu U^2U',\label{ODE1}\\
U(x_0)&=U_0>0,\; U'(x_0)=U_1\in \R,\; U''(x_0)=U_2\in \R\label{ODE2}\,,
\end{align}
\end{subequations}
where $n>1,\; \mu>0,\; x_0\in \R.$
Then there exists a unique maximal solution $U=U(x)>0$ of \eqref{ODE}, defined on some interval $(x_*,x^*)$ with $-\infty \leq x_*<x_0<x^*\leq\infty$.
\end{lemma}

The proof of this lemma is postponed to the Appendix \ref{odeexistence}. As an application of Lemma \ref{ODE-max} we have the following.
\begin{proposition}
\label{Hbmumax}
The function $H_{b,\mu}$ given in \eqref{e5.01} can be extended to a smooth solution of \eqref{tfeselfts}-\eqref{bdryselfts1} on a maximal interval $(0,x^*_{b,\mu})$ with
\begin{equation}
\label{e5.02}
H_{b,\mu}>0 \;\mbox{in}\; (0,x^*_{b,\mu})\;\; \mbox{and}\;\; x^*_{b,\mu}\leq\infty.
\end{equation}
\end{proposition}
\begin{proof}
Let $U=H_{b,\mu},\; x_0={1\over 2}\hat{x}_{b,\mu}$ where $H_{b,\mu}$ is given by \eqref{e5.01} and $\hat{x}_{b,\mu}$  is given by \eqref{defhatx}. Since $H_{b,\mu}$ satisfies \eqref{ODE1} on $(0, \hat{x}_{b,\mu})$ and $H_{b,\mu}(0)=0$, then $U$ satisfies \eqref{ODE1}-\eqref{ODE2} with $U(x_0)=H_{b,\mu}(x_0)>0,\; U'(x_0)=H_{b,\mu}'(x_0),\; U''(x_0)=H_{b,\mu}''(x_0).$ By Lemma \ref{ODE-max}, $U=H_{b,\mu}$ can be extended to a smooth solution of \eqref{tfeselfts}-\eqref{bdryselfts1} on a maximal interval $(0,x^*_{b,\mu})$.
\end{proof}
Our goal  is to show the existence of a solution satisfying \eqref{bdryselfts2} and \eqref{bdryselfts3} as well. To fulfill condition \eqref{bdryselfts2} we shoot with the parameter $b$. Thus, we obtain a solution $H_{\overline{b}(\mu),\mu}$ of (\ref{tfeselfts}) which satisfies  (\ref{bdryselfts1}) and  \eqref{bdryselfts2}. We conclude by a shooting argument with $\mu$ to fulfill condition \eqref{bdryselfts3}. For both, the following expansions are essential:
\begin{lemma}
\label{Lexpansions}
Let $H_{\mathrm{TW}}$ be the traveling-wave solution of
\eqref{freebdrytw} given by \eqref{Htw}, and $H_{b,\mu}$ the function defined by equation
\eqref{e5.01}. There exists $\eps_0>0$ such that the following holds.
\begin{eqnarray}
\nonumber
\partial_x^k(H_{b,\mu} - H_{\mathrm{TW}})(x)&=&\frac{A^{1-\nu/3}}{p(1)}\big(1+O(\eps)\big)
\partial_x^kx^{\nu+1}-bA^{-\nu/3}\big(1+O(\eps)\big) \partial_x^k x^{\nu+\beta}\\&& \label{DVA}+
\frac{\mu\nu A^{-\nu}}{p(\gamma)}\big(1+O(\eps)\big)\partial_x^kx^{\nu+\gamma},
\end{eqnarray}
\begin{equation}
\hspace{-2,5cm}\partial_x^k\partial_b H_{b,\mu}(x)=-A^{-\nu/3}\big(1+O(\eps)\big) \partial_x^k x^{\nu+\beta},
\label{VPB}
\end{equation}
and
\beq
\label{VPMU}
\hspace{-2,7cm}\partial_x^k\partial_\mu H_{b,\mu}(x)=\frac{\nu A^{-\nu}}{p(\gamma)}\big(1+O(\eps)\big)\partial_x^kx^{\nu+\gamma},
\eeq
for $k\in\{0,1,2,3\}$, $0\leq\eps\leq\eps_0$ and $0\leq x\leq \hat{x}_{b,\mu}(\eps).$
\end{lemma}
We point out that $O(\varepsilon)$ means a generic function $f(x,\varepsilon)$ with $|f(x,\varepsilon)|\lesssim \varepsilon$ for $x$ near $0$.
\begin{proof} We have that $\left(H_{b,\mu}-H_{\mathrm{TW}}\right)(x)=A^{-\nu/3} x^\nu u_{b,\mu}(x)$, where $u_{b,\mu}(x)=\bar{u}(x,b x^\beta, \mu x^\gamma).$
Since $\bar{u}$ satisfies the equations \eqref{Nonlinear1}-\eqref{Nonlinear2}-\eqref{source}, and using the fact that $\partial_1p(\bar{\bD})\bar{u}=p(\bar{\bD}+1)\partial_1\bar{u}$ and $\partial_3 p(\bar{\bD})\bar{u}=p(\bar{\bD}+\gamma)\partial_3\bar{u}$ we get
$$
\partial_1\bar{u}(0,0,0)=\frac{A}{p(1)}>0,\quad \partial_3\bar{u}(0,0,0)=\frac{\nu A^{-{2\over 3}\nu}}{p(\gamma)}>0\,.
$$
Hence
\begin{equation}
\label{DASYMP}
\bar{u}(x_1,x_2,x_3)=\frac{A}{p(1)}\big(1+O(\eps)\big) x_1-\big(1+O(\eps)\big) x_2 +
\frac{\nu A^{-{2\over 3}\nu}}{p(\gamma)}\big(1+O(\eps)\big)x_3.
\end{equation}
By definition of $\bar{\bD}$, we have that
\begin{equation}
\label{e5commutation}
\partial_x^k\left(x^\nu u_{b,\mu}(x)\right)=x^{\nu-k}\displaystyle\prod_{j=0}^{k-1}\left(\bar{\bD}+\nu-j\right)\bar{u}(x,b x^\beta,\mu x^\gamma).
\end{equation}
We also have
\begin{subequations}\label{bmuvariation}
\begin{align}
&\partial_b u_{b,\mu}(x)=\frac{1}{b}x_2\partial_2\bar{u}(x,bx^\beta,\mu x^\gamma),
\label{bmuvariation1}\\
&\partial_\mu u_{b,\mu}(x)=\frac{1}{\mu}x_3\partial_3\bar{u}(x,bx^\beta,\mu x^\gamma).\label{bmuvariation2}
\end{align}
\end{subequations}
The analyticity of $\bar{u}$ and \eqref{DASYMP} imply
\begin{eqnarray*}
\bar{\bD}^k\bar{u}&=&\frac{A}{p(1)}\big(1+O(\eps)\big) x_1-\beta^k\big(1+O(\eps)\big) x_2 +
\gamma^k\frac{\nu A^{-{2\over 3}\nu}}{p(\gamma)}\big(1+O(\eps)\big)x_3,
\\
x_2\partial_2\left(\bar{\bD}^k\bar{u}\right)&=&-\beta^k\big(1+O(\eps)\big) x_2 ,\\
x_3\partial_3\left(\bar{\bD}^k\bar{u}\right)&=&\gamma^k\frac{\nu A^{-{2\over 3}\nu}}{p(\gamma)}\big(1+O(\eps)\big)x_3.
\end{eqnarray*}

It follows from \eqref{DASYMP}-\eqref{e5commutation} and
$D u_{b,\mu}(x)=\bar{\bD}\bar{u}(x,bx^\beta, \mu x^\gamma)$ that
\begin{eqnarray}
\nonumber \partial_x^k(H_{b,\mu} -
H_{\mathrm{TW}})(x)&=&\partial_x^k A^{-\nu/3} x^\nu u_{b,\mu}(x)\\
&=& \nonumber
A^{-\nu/3}x^{\nu-k}\displaystyle\prod_{j=0}^{k-1}\left(\bar{\bD}+\nu-j\right)\bar{u}(x,b
x^\beta,\mu x^\gamma)\\ \nonumber
&=&A^{-\nu/3}x^{\nu-k}\displaystyle\prod_{j=0}^{k-1}\left({\bD}+\nu-j\right)\Big(\frac{A}{p(1)}
\big(1+O(\eps)\big) x\\ && \nonumber  -b\big(1+O(\eps)\big) x^\beta+
\frac{\mu\nu
A^{-{2\over 3}\nu}}{p(\gamma)}\big(1+O(\eps)\big)x^\gamma\Big)\\
\nonumber
&=&A^{-\nu/3}\partial_x^kx^{\nu}\Big(\frac{A}{p(1)}\big(1+O(\eps)\big)
x -b\big(1+O(\eps)\big) x^\beta \\ && \nonumber + \frac{\mu\nu
A^{-{2\over 3}\nu}}{p(\gamma)}\big(1+O(\eps)\big)x^\gamma\Big)\\
\nonumber &=&\frac{A^{1-\nu/3}}{p(1)}\big(1+O(\eps)\big)
\partial_x^kx^{\nu+1} -bA^{-\nu/3}\big(1+O(\eps)\big) \partial_x^k
x^{\nu+\beta} \\&& \nonumber + \frac{\mu\nu
A^{-\nu}}{p(\gamma)}\big(1+O(\eps)\big)\partial_x^kx^{\nu+\gamma}.
\end{eqnarray}
This proves \eqref{DVA}. We easily deduce from \eqref{DVA}, the formulas \eqref{VPB} and \eqref{VPMU}.
\end{proof}

In the following Lemma and Proposition, $\mu$ is assumed to be a fixed positive real number. A key lemma is the following:
\begin{lemma}
\label{L5.01}
Let $\mu>0$ be fixed and $H_{\mathrm{TW}}$ be the traveling-wave solution of
\eqref{freebdrytw} given by \eqref{Htw}. The function $H_{b,\mu}$ defined by equation \eqref{e5.01} satisfies:
\begin{itemize}
\item[(i)]  $\partial_x^k H_{0,\mu}(x)>\partial_x^k H_{\mathrm{TW}}(x)$ for $k=0,1,2,3$ and $x\in (0,x^*_{0,\mu})$. In particular $H_{0,\mu}$ does not reach $0$.
\item[(ii)] $\partial_b\partial_x^k H_{b,\mu}(x)\leq 0$ for $k=0,1,2,3$ and $x\in [0, \hat{x}_{b,\mu})$.
\item[(iii)] $x_{b,\mu}^*\to 0$ as $b\to\infty$.
\end{itemize}
\end{lemma}
\begin{proof} From \eqref{DVA} we have, for $\varepsilon>0$ sufficiently small, that
\beq
\label{order1}
\partial_x^k H_{0,\mu}>\partial_x^k H_{\mathrm{TW}} \quad\mbox{on}\quad (0, \hat{x}_{0,\mu}(\eps)]\quad\mbox{for}\quad k=0,1,2.
\eeq
From equations \eqref{tfeselfts}-\eqref{tfetw} and the fact that $\mu>0,$ we have
\begin{eqnarray}
\nonumber
\left(H_{0,\mu}-H_{\mathrm{TW}}\right)'''&=&{H_{0,\mu}^{n-1}-H_{\mathrm{TW}}^{n-1}\over H_{0,\mu}^{n-1}H_{\mathrm{TW}}^{n-1}}+{x\over H_{0,\mu}^{n-1}}+\mu H_{0,\mu}^2H_{0,\mu}'\\
\nonumber
&>& {H_{0,\mu}^{n-1}-H_{\mathrm{TW}}^{n-1}\over H_{0,\mu}^{n-1}H_{\mathrm{TW}}^{n-1}}+\mu H_{0,\mu}^2H_{0,\mu}'\\
\nonumber
&>& {H_{0,\mu}^{n-1}-H_{\mathrm{TW}}^{n-1}\over H_{0,\mu}^{n-1}H_{\mathrm{TW}}^{n-1}\left(H_{0,\mu}-H_{\mathrm{TW}}\right)}\,\left(H_{0,\mu}-H_{\mathrm{TW}}\right)
\\ \label{compb=0} &+&\mu H_{0,\mu}^2\left(H_{0,\mu}-H_{\mathrm{TW}}\right)'.
\end{eqnarray}
The first assertion (i) follows from \eqref{order1}, \eqref{compb=0} and Corollary \ref{comparstr}.

 We now turn to the proof of (ii). From \eqref{VPB} we have, for $\varepsilon>0$ sufficiently small, that
\beq
\label{order1bis}
\partial_b\partial_x^k H_{b,\mu}<0 \quad\mbox{on}\quad (0, \hat{x}_{b,\mu}(\eps)]\quad\mbox{for}\quad k=0,1,2.
\eeq
Differentiating equation \eqref{tfeselfts} with respect to $b$ yields
\begin{equation}
\label{nouveaunumber}
G''' ={(n-1)(1-x)\over H^n}\,G+2\mu H H' \,G+\mu H^2\,G',
\end{equation}
where $G=\partial_b H$ and $H=H_{b,\mu}$. By \eqref{DVA}  the coefficients in the previous equation on $G$ are positive.  The assertion (ii) follows by the ordering \eqref{order1bis}, the equation \eqref{nouveaunumber} and Corollary \ref{B2}.

Finally, we turn to prove (iii). For $b\geq \max\left(\varepsilon^{1-2\beta},\mu^{\beta/\gamma}\varepsilon^{1-2\beta/\gamma}\right),$ we have $\hat{x}_{b,\mu}(\varepsilon)=\left(\varepsilon/b\right)^{1/\beta}.$ Hence, it follows from the expansion \eqref{DVA} with $b$ sufficiently large, and the fact that $\beta<1$
\begin{eqnarray}
\nonumber
H_{b,\mu}-H_{\mathrm{TW}}&\leq & 0\\
\label{undershooting}
\left(H_{b,\mu}-H_{\mathrm{TW}}\right)'&\leq & 0\\
\nonumber
\left(H_{b,\mu}-H_{\mathrm{TW}}\right)''&\lesssim & -b^{{2-\nu\over \beta}}\varepsilon^{1+ {\nu-2\over \beta}}
\end{eqnarray}
 at $x=\left(\varepsilon/b\right)^{1/\beta}.$
 Also, using the monotonicity in $b,$ we obtain for $x\leq 1,$
 \begin{eqnarray*}
 \left(H_{b,\mu}-H_{\mathrm{TW}}\right)'''&=&{-1+x\over H_{b,\mu}^{n-1}}+{1\over H_{\mathrm{TW}}^{n-1}}+\mu H_{b,\mu}^2H_{b,\mu}'\\
 &\leq &{-1+x\over H_{0,\mu}^{n-1}}+{1\over H_{\mathrm{TW}}^{n-1}}+\mu H_{b,\mu}^2H_{b,\mu}'\\&\leq &(1-x){H_{0,\mu}^{n-1}-H_{\mathrm{TW}}^{n-1}\over H_{0,\mu}^{n-1}H_{\mathrm{TW}}^{n-1}}+{x\over H_{\mathrm{TW}}^{n-1}}+\mu H_{b,\mu}^2H_{b,\mu}'\\
 &\leq&(n-1)(1-x){H_{0,\mu}-H_{\mathrm{TW}}\over H_{\mathrm{TW}}^n}+{x\over H_{\mathrm{TW}}^{n-1}}+\mu H_{b,\mu}^2H_{b,\mu}',
 \end{eqnarray*}
 where we have used (i) with $k=0$ and the inequality
 $$
 {X^\alpha-Y^\alpha\over \left(X Y\right)^\alpha}\leq {\alpha\over Y^{\alpha +1}}\left(X-Y\right),\quad \alpha>0\quad\mbox{and}\quad 0< Y<X\,.
 $$
 By \eqref{Htw}, \eqref{e5.01} and \eqref{DASYMP}, we have that
 \begin{eqnarray*}
(1-x){H_{0,\mu}-H_{\mathrm{TW}}\over H_{\mathrm{TW}}^n}\;&\sim&\;{x^{\nu+1}\over x^{n\nu}}=x^{\nu-2}\quad\mbox{as}\quad x\searrow 0,\\
{x\over H_{\mathrm{TW}}^{n-1}}\;&\sim&\;{x\over x^{(n-1)\nu}}=x^{\nu-2}\quad\mbox{as}\quad x\searrow 0,\\
H_{b,\mu}^2H_{b,\mu}'\;&\sim&\;x^{3\nu-1}=o(x^{\nu-2})\quad\mbox{as}\quad x\searrow 0,
\end{eqnarray*}
 and since $ \left(H_{b,\mu}-H_{\mathrm{TW}}\right)'''$ is regular for $x>0$, we conclude that
 \begin{equation}
 \label{xnu-2}
 \left(H_{b,\mu}-H_{\mathrm{TW}}\right)''' \lesssim x^{\nu-2},
 \end{equation}
  $\mbox{ for } \; x\in \left(\left(\varepsilon/b\right)^{1/\beta}, \,\min\{1, x_{b,\mu}^*\}\right),\; b\geq \max\left(\varepsilon^{1-2\beta},\mu^{\beta/\gamma}\varepsilon^{1-(2\beta/\gamma)}\right).$

 The Taylor expansion of $H_{b,\mu}-H_{\mathrm{TW}}$ around $\hat{x}_{b,\mu}(\varepsilon)$ reads
 \begin{eqnarray*}
 \left(H_{b,\mu}-H_{\mathrm{TW}}\right)(x)&=& \left(H_{b,\mu}-H_{\mathrm{TW}}\right)(\hat{x}_{b,\mu}(\varepsilon))\\&+&\left(x-\hat{x}_{b,\mu}(\varepsilon)\right) \left(H_{b,\mu}-H_{\mathrm{TW}}\right)'(\hat{x}_{b,\mu}(\varepsilon))\\&+&{1\over 2}\left(x-\hat{x}_{b,\mu}(\varepsilon)\right)^2 \left(H_{b,\mu}-H_{\mathrm{TW}}\right)''(\hat{x}_{b,\mu}(\varepsilon))\\&+&
 {1 \over 2}\int_{\hat{x}_{b,\mu}(\varepsilon)}^x\;\left(x-y\right)^2 \left(H_{b,\mu}-H_{\mathrm{TW}}\right)'''(y)\,dy.
 \end{eqnarray*}

 Using \eqref{undershooting} and \eqref{xnu-2} for $x\geq \hat{x}_{b,\mu}$ close to $\hat{x}_{b,\mu},$ we get
 \begin{equation*}
 H_{b,\mu}(x) \leq  H_{\mathrm{TW}}(x)-c_1 b^{{2-\nu\over \beta}}\varepsilon^{1+ {\nu-2\over \beta}}\left(x-\hat{x}_{b,\mu}(\varepsilon)\right)^2+c_2 x^{\nu+1}.
 \end{equation*}
 It follows, since $\nu<2,$ that for $b$ sufficiently large, the right hand side of the previous inequality is negative. This completes the proof of Part (iii). This finishes the  proof of the lemma.
  \end{proof}
A consequence of the previous lemma is:
\begin{proposition}
\label{P5.01}
Let $\mu>0$ be fixed. Then there exists $\bar{b}(\mu)>0$ such that the function $H_{\bar{b}(\mu),\mu}$ satisfies \eqref{tfeselfts}--\eqref{bdryselfts2}. Moreover, $H_{\bar{b}(\mu),\mu}'>0$ on $(0,1)$.
\end{proposition}

\begin{proof}
We have $H_{b,\mu}'(x)>0$ for $x$ near $0$ by \eqref{DVA}, and, for $b$ {sufficiently} large, $H_{b,\mu}'$ is negative somewhere by Lemma \ref{L5.01} Part (iii). Hence, there exists $\bar{x}_{b,\mu}$ such that $H_{b,\mu}'(\bar{x}_{b,\mu})=0$. Define
$$
\mathcal{B}=\Big\{\;b>0;\;\; H_{b,\mu}'(x)=0\quad\mbox{ for some}\quad x\in (0,1]\cap (0,x_{b,\mu}^*)\;\Big\}.
$$
Let $\bar{b}(\mu)=\inf\mathcal{B}$ which is well defined. Part $(i)$ of Lemma \ref{L5.01} ensures that $\bar{b}(\mu)>0$. Moreover, by continuous dependence on the parameter $b$, $\bar{b}(\mu)\in\mathcal{B}$.

 To conclude we will prove that $\bar{\bar{x}}_{\bar{b}(\mu)}=1,$ where for $b\in\mathcal{B}$, $\bar{\bar{x}}_{b}$ stands for the first zero of $H_{b,\mu}'$. Assume by contradiction that $\bar{\bar{x}}_{\bar{b}(\mu)}<1.$ Then, since $H_{\bar{b}(\mu),\mu}'(\bar{\bar{x}}_{\bar{b}(\mu)})=0,$ we get $H_{\bar{b}(\mu),\mu}'''(\bar{\bar{x}}_{\bar{b}(\mu)})=
(-1+\bar{\bar{x}}_{\bar{b}(\mu)})H_{\bar{b}(\mu),\mu}^{1-n}(\bar{\bar{x}}_{\bar{b}(\mu)})<0.$ Hence $H_{\bar{b}(\mu),\mu}'''<0$ in some neighborhood of $\bar{\bar{x}}_{\bar{b}(\mu)}$ and  $H_{\bar{b}(\mu),\mu}''$ is decreasing. Moreover, the fact that $H_{\bar{b}(\mu),\mu}'(\bar{\bar{x}}_{\bar{b}(\mu)})=0$ and $H_{\bar{b}(\mu),\mu}'>0$ on $(0, \bar{\bar{x}}_{\bar{b}(\mu)})$ implies that $H_{\bar{b}(\mu),\mu}''(\bar{\bar{x}}_{\bar{b}(\mu)})\leq 0.$ Using the fact that $H_{\bar{b}(\mu),\mu}''$ is decreasing, we deduce that
$$H_{\bar{b}(\mu),\mu}''<0, \mbox{ on } (\bar{\bar{x}}_{\bar{b}(\mu)}, \bar{\bar{x}}_{\bar{b}(\mu)}+\eta),$$
 for some  $\eta>0.$ Then
 $$H_{\bar{b}(\mu),\mu}'<0, \mbox{ on } (\bar{\bar{x}}_{\bar{b}(\mu)}, \bar{\bar{x}}_{\bar{b}(\mu)}+\eta).$$
This contradicts the definition of $\bar{b}(\mu)$ and proves that $\bar{\bar{x}}_{\bar{b}(\mu)}=1.$ It follows that $H_{\bar{b}(\mu),\mu}$ is the desired solution satisfying (\ref{tfeselfts})-(\ref{bdryselfts1})-\eqref{bdryselfts2}.
\end{proof}
In the sequel, we will denote
\begin{equation}
\label{Hmu}
H_\mu:=H_{\bar{b}(\mu),\mu},
\end{equation}
where, for any $\mu>0$, $H_{\bar{b}(\mu),\mu}$ is given by Proposition \ref{P5.01}.
\begin{proposition}
\label{p5.06}
There exist two positive constants $C_n$ and {$D_n$} depending only on $n$ such that for all $\mu >0$, we have
\bq
\label{Ineq1}
\max\left(H_\mu(1)^{n/4}, \sqrt{\mu} H_\mu(1)^{1+n/2}\right)\geq {D_n},
\eq
\begin{equation}
\label{Ineq2}
-H_\mu''(1)=|H_\mu''(1)|\leq \,C_n\,\sqrt{n+4}\,H_\mu(1)^{1-n/2},
\end{equation}
where $H_\mu$ is given by \eqref{Hmu}.
\end{proposition}
\begin{proof}
Let $\mu>0$ and $a>0$ be defined by \eqref{mu}. Then
\begin{equation}
\label{calHmu}
{\mathcal H}_\mu(y)=(n+4)^{- 1/n}\,a^{4/n}\, H_\mu(1-\frac{y}{a}),
\end{equation}
defined for $y\in [0, a)$ and extended to $(-a,a)$ by evenness solves \eqref{tfeself}-\eqref{bdryself2}. Using \cite[Lemma 3.3, p. 750]{B}, \cite[(3.32), p. 754]{B} together with the fact that
$$
{\mathcal H}_\mu(0)=(n+4)^{- 1/n}\,a^{4/n}\, H_\mu(1),\;\;\;{\mathcal H}_\mu''(0)=(n+4)^{- 1/n}\,a^{-2+4/n}\, H_\mu''(1)
$$
we obtain \eqref{Ineq1} and \eqref{Ineq2}. This finishes the proof of Proposition \ref{p5.06}.
\end{proof}
\begin{remark}
The constant $C_n$ appearing in \eqref{Ineq2} is as in \cite[(3.32), p. 754]{B} while the constant {$D_n$} can be taken as
$$
{D_n}=\frac{(n+3)^{-1/2} (n+4)^{-1/4}}{\sqrt{48}}.
$$
\end{remark}
To satisfy \eqref{bdryselfts3} it suffices to prove the following.
\begin{proposition}
\label{barmu}
There exists $\bar{\mu}>0$ such that $H_{\bar{\mu}}$ satisfies \eqref{bdryselfts3}, namely
\begin{equation}
\label{1.7d}
\int_{0}^1 H_{\bar{\mu}}(x) dx = {\sqrt{n+3}\over
2\sqrt{\bar{\mu}}}\M,
\end{equation}
where $\M>0$ is fixed by \eqref{bdry3}.
\end{proposition}
\begin{proof}
It suffices to show that the map
$$
\mathcal{M} : 0\leq \mu \longmapsto \mathcal{M}(\mu):= {2\sqrt{\mu}\over \sqrt{n+3}} \int_{0}^1 H_{\mu}(x) dx
$$
satisfies $\mathcal{M}([0,\infty))=[0,\infty)$. Note that  $\mathcal{M}$ is continuous with respect to $\mu$ and $\mathcal{M}(0)=0$. To conclude the proof, we will show that there exists a sequence $(\mu_j)$ such that
\begin{equation}
\label{muj}
\mathcal{M}(\mu_j) \to \infty \quad\mbox{as}\quad j\to \infty.
\end{equation}
The proof of \eqref{muj} will be done in two steps.\\

\noindent{\bf Step 1.} Define
\begin{equation}
\label{alphamu}
\alpha(\mu)=(n+4)^{-\frac{1}{n+4}}\, (n+3)^{-\frac{2}{n+4}}\, \mu^{\frac{2}{n+4}}\, H_\mu(1).
\end{equation}
We claim that
\begin{equation}
\label{supalpha}
\sup_{\mu>0}\, \alpha(\mu)=\infty.
\end{equation}
By contradiction: assume that $\alpha(\mu)\lesssim 1$. Then
$$
H_\mu(1)\lesssim \mu^{-\frac{2}{n+4}}.
$$
Therefore
\begin{eqnarray*}
H_\mu(1)^{n/4}&\lesssim& \mu^{-\frac{n}{2(n+4)}},\\
\sqrt{\mu}\, H_\mu(1)^{1+n/2}&\lesssim& \mu^{-\frac{n}{2(n+4)}}.
\end{eqnarray*}
By using \eqref{Ineq1}, we should have $ 1\lesssim \mu^{-\frac{n}{2(n+4)}}$ for all $\mu>0$. This leads to a contradiction for $\mu$ large and the proof of the claim follows.

From \eqref{supalpha} we deduce that there exists a sequence $(\mu_j)$ such that
\begin{equation}
\label{alphaj}
\alpha_j:=\alpha(\mu_j)\to \infty\quad \mbox{as}\quad j\to \infty.
\end{equation}

\noindent{\bf Step 2.} Let ${\mathbf v}_j=\mathbf{v}_j(y)$ be the solution of
\begin{equation}\label{Vj}
 \left\{
\begin{array}{ll}
{\mathbf v}_j''&=-2C_n\, \alpha_j^{1-n/2}+\alpha_j^2(n+3)\left({\mathbf v}_j-\alpha_j\right) \;\;\mbox{for}\;\; y>0, \\
{\mathbf v}_j(0)&=\alpha_j,\quad {\mathbf v}_j'(0)=0,
\end{array}
\right.
\end{equation}
where $C_n$ is as in \eqref{Ineq2}. The solution of \eqref{Vj} is given explicitly by

\begin{equation*}
\mathbf{v}_j(y)=\frac{1}{(n+3)\alpha_j^{1+n/2}}\,
\bigg[ 2C_n+(n+3)\alpha_j^{2+n/2}-2C_n\,\cosh\left(\alpha_j\sqrt{n+3}\, y\right)\bigg].
\end{equation*}
Using same arguments as in the proof of \cite[Lemma 3.7, p. 754]{B} we have that
\begin{equation*}
\mathcal{H}_{\mu_j}(y)\geq \mathbf{v}_j(y)\;\;\mbox{for}\;\; y>0,
\end{equation*}
where $\mathcal{H}_{\mu_j}$, given by \eqref{calHmu}, is extended by zero outside its support. Let $\tilde{y}_j$ be such that $\mathbf{v}_j(\tilde{y}_j)=\frac{\alpha_j}{2}$. Then
\begin{eqnarray*}
\int_0^{a_j}\,\mathcal{H}_{\mu_j}(y)\,dy&\geq& \int_0^{\tilde{y}_j}\,\mathbf{v}_j(y)\,dy\\
&\geq& \frac{\alpha_j}{2}\,\tilde{y}_j=\frac{1}{2\sqrt{n+3}}\,\arg\cosh\,\left(1+\frac{n+3}{4C_n}\alpha_j^{2+n/2}\right).
\end{eqnarray*}
Using \eqref{alphaj} we deduce that
\begin{equation*}
\int_0^{a_j}\,\mathcal{H}_{\mu_j}(y)\,dy \to\infty \quad\mbox{as}\quad j\to\infty.
\end{equation*}
Now we can conclude the proof of \eqref{muj}. Indeed, we have that
$$
\mathcal{M}(\mu_j)=2\int_0^{a_j}\,\mathcal{H}_{\mu_j}(y)\,dy.
$$
It follows that, for $\M>0$ given by \eqref{bdry3}, there exists $\bar{\mu}>0$ such that $\mathcal{M}(\bar{\mu})=\M$. This finishes the proof of Proposition \ref{barmu}.
\end{proof}
\begin{proof}[Proof of Theorem \ref{th1}] The proof of Part (i) follows from Proposition \ref{p2},  \eqref{e5.01} and Proposition \ref{P5.01}. The proof of Part (ii) follows by Proposition \ref{barmu} with $H_{\bar{\mu}}=H_{\bar{b}(\bar{\mu}),\bar{\mu}}.$ This completes the proof of Theorem \ref{th1}.
\end{proof}

\section{Conclusions}
We consider self-similar source-type solutions $H$ for the thin-film equation with a regularizing second order term and with mobility exponent $n\in ({3\over2}, 3)$ in dimension one. We show the existence of a solution having the behavior: $H(x)=H_{\mathrm{TW}}(x)\left(1+v(x,x^\beta,x^\gamma)\right)$, where $H_{\mathrm{TW}}$ is the traveling-wave, $\beta\in(0,1)$, $\gamma=2+\frac{6}{n}$, and $v(x_1,x_2,x_3)$ is analytic near $(0,0,0)$ with $v(0,0,0)=0$, $\partial_2 v(0,0,0)<0$. This improves the previously published results {\cite{B}} about qualitative behavior of the solution nearby the interface. The previous asymptotic shows that the source-type solution for the thin-film equation with gravity is an analytic function in the  three spacial variables $(x_1,x_2,x_3)$ where $x_1:=x$, $x_2:=x^\beta$ and  $x_3:=x^\gamma$. The third variable is new, unless $n=2$, with respect to the  known expansion for the standard thin-film equation.

This shows the effect of the gravity on the expansion of source type solutions. We expect this to be the generic behavior of solutions of the thin-film equation with gravity \eqref{tfe} and to be helpful for the well-posedness for \eqref{tfe}. In fact, it is shown in \cite{GGKO} that the expansion, given in \cite{GGO}, of the source type solution for the standard thin-film equation \eqref{stfe} has an effect on the behavior of the solutions. Also, this expansion was useful in \cite{GGKO} to obtain a well-posedness result for \eqref{stfe}. See also \cite{BF1990, G, GM2018, Knup2011, km.2013, km.2012} for well-posedness for \eqref{stfe}.

Source-type self-similar solutions are useful to describe the long time behavior of a large class of solutions to thin-film equations. We expect that the source-type self-similar solutions we construct here will attract for large time some global solutions.  This has been done for \eqref{stfe} in
\cite{CU2005, CU2007, CU2014, CT2002, GGO2016, GKO, go.2002, G}. See also \cite{GIM, MMT2021} for other asymptotic behavior.\\

\begin{appendix}
\section{Existence and uniqueness for ODE}
\label{odeexistence}
Consider the ordinary differential equation
\begin{equation}
\label{a1}
y'=f(x,y)
\end{equation}
where $f:E\subset \R\times \R^d\to \R^d$, with $E$ an open set. We recall the following existence and uniqueness results  for \eqref{a1}.
\begin{theorem}{\cite[Theorem 3.1, p. 18]{Hale}}\label{Hale} If $f(x,y)$ is continuous in $E$ and locally Lipschitz with respect to $y$ in $E,$ then for any $(x_0,y_0)\in E,$ there exists a unique solution $y(x)$ of \eqref{a1} satisfying $y(x_0)=y_0.$
\end{theorem}
We also recall the following extension result.

\begin{theorem}{\cite[Theorem 3.1, p. 12]{Hart}}\label{Hart} Let $f(x,y)$ be continuous on an open set $E$ and let $y(x)$ be a solution of \eqref{a1} on some interval. Then $y(x)$ can be extended (as a solution) over a maximal interval of existence $(x_*,x^*).$ Also, if $(x_*,x^*)$ is a maximal interval of existence, then $y(x)$ tends to the boundary $\partial E$ of $E$ as $x\to x_*$ and $x\to x^*.$
\end{theorem}

We now give the proof of Lemma \ref{ODE-max}.

\begin{proof}[Proof of Lemma \ref{ODE-max}.]  Let $E=\R\times \Big((0,\infty)\times \R^2\Big)$ and $f : E \to \R^3$ given by
$$
f(x,y)=\Big(y_2,\;y_3,\;(x-1)y_1^{1-n}+\mu y_1^2y_2\Big),\;\;y=(y_1,y_2,y_3).
$$
Clearly $f$ is continuous in $E$ and locally Lipschitz with respect to $y.$ The problem \eqref{ODE1}-\eqref{ODE2} is equivalent to
\begin{eqnarray*}
y'&=&f(x,y),\; \\
 y(x_0)&=&(U_0,U_1,U_2)\in (0,\infty)\times \R^2,
\end{eqnarray*}
where $ y(x)=\Big(U(x),\,U'(x),\,U''(x)\Big).$ Using Theorems \ref{Hale}-\ref{Hart}, we obtain the existence of a unique maximal solution on $(x_*,x^*)$ with $-\infty\leq x_*<x_0<x^*\leq\infty.$
\end{proof}
\section{Useful tools} \label{touls}
In this appendix we recall some known facts for ordinary differential equations. We have the following comparison result.
\begin{proposition}
\label{compar1}
Assume that the function $y: [a,b]\to \R$ satisfies the ordinary differential inequality
\begin{equation}
\label{ineq1}
y'''(x)\geq A(x)y(x)+B(x)y'(x)+C(x)y''(x),\; \; a\leq x\leq b,
\end{equation}
where $A,\;B,\; C$ are nonnegative continuous functions.

If $y^{(k)}(a)\geq 0,\; k=0,\; 1,\; 2$ then
\begin{equation}
\label{B-22}
y'''(x)\geq 0,\;  \; a\leq x\leq b.
\end{equation}
\end{proposition}

To prove Proposition \ref{compar1}, we need to introduce the  type $K$ function.
\begin{definition}{\cite[p. 27]{Coppel}}
\label{defvector}
Let $Y=(y_1,y_2,y_3),\; Z=(z_1,z_2,z_3)$ be two vectors in $\R^3.$ We say that $Y\geq Z$, if $y_i\geq z_i,\;$ for all $i=1,\; 2,\; 3$.
\end{definition}

\begin{definition}{\cite[p. 27]{Coppel}}
\label{K-type}
A vector function $f=(f_1,f_2,f_3)$ of a vector variable $Y=(y_1,y_2,y_3)$ will be said to be of type K in a set $S$ if for each $i=1,\; 2,\; 3$ we have $f_i(Y)\leq f_i(Z)$ for any two vectors $Y=(y_1,y_2,y_3),\; Z=(z_1,z_2,z_3)$ in $S$ with $y_i=z_i$ and $y_j\leq z_j$ ($j=1,\; 2,\; 3;\; j\neq i$).
\end{definition}

\begin{proof}[Proof of Proposition \ref{compar1}] Let $Y=(y_1,y_2,y_3)$ and $f(x,Y)=(y_2,y_3, A(x)y_1+B(x)y_2+C(x)y_3).$ Then using Definition \ref{defvector}, the differential inequality reads
$$Y'(x)\geq f(x,Y(x)),\; a\leq x\leq b,$$ where $Y=(y,y',y'').$ Since $A,\; B,\; C$ are nonnegative then by Definition \ref{K-type}, $f$ is type K. Using \cite[Theorem 10, p. 29]{Coppel}, and since $Y(a)\geq 0=(0,0,0)$ we get $Y(x)=(y(x),y'(x),y''(x))\geq (0,0,0), \; a\leq x\leq b.$ Using the differential inequality and the fact that $A,\; B,\; C$ are nonnegative we get \eqref{B-22}. This completes the proof of the proposition.
\end{proof}

From Proposition \ref{compar1} we deduce the following results.
\begin{corollary}
\label{comparstr}
Assume that the function $y: [a,b]\to \R$ satisfies the ordinary differential inequality
\begin{equation}
\label{ineq2-1} y'''(x)> A(x)y(x)+B(x)y'(x)+C(x)y''(x),\; \; a\leq x\leq b,
\end{equation}
where $A,\;B,\; C$ are positive continuous functions.

If $y^{(k)}(a)> 0,\; k=0,\; 1,\; 2$ then
\begin{equation}
\label{ineq2}y'''(x)> 0,\;  \; a\leq x\leq b.
\end{equation}
\end{corollary}
\begin{proof}[Proof of Corollary \ref{comparstr}] Using Proposition \ref{compar1}, we deduce that $y^{(k)}(x)\geq 0,\;  \; a\leq x\leq b,\; k=0,\; 1,\; 2,\; 3.$ Then the desired inequality \eqref{ineq2} follows immediately from \eqref{ineq2-1}.
\end{proof}

\begin{corollary}
\label{B2}
Assume that the function $y: [a,b]\to \R$ satisfies the ordinary differential equation
\begin{equation}
\label{eq22} y'''(x)= A(x)y(x)+B(x)y'(x)+C(x)y''(x),\; \; a\leq x\leq b,
\end{equation}
where $A,\;B,\; C$ are positive continuous functions.
If $y^{(k)}(a)< 0,\; k=0,\; 1,\; 2$ then
\begin{equation}
\label{ineq2-2}y'''(x)\leq 0,\;  \; a\leq x\leq b.
\end{equation}
\end{corollary}

\begin{proof}[Proof of Corollary \ref{B2}] Put $z=-y$. Then $z$ satisfies the assumptions in Proposition \ref{compar1}. Hence we obtain \eqref{ineq2-2}.
\end{proof}
\end{appendix}


\section*{ Acknowledgements}
\noindent The authors are much obliged to M. V. Gnann and N. Masmoudi for interesting discussions. {The authors thank the reviewers for the careful reading of the manuscript and helpful comments.}

\end{document}